\documentclass[11pt]{article}

\usepackage{amsmath,amssymb,amsthm}
\usepackage{mathrsfs}
\usepackage{enumerate}
\usepackage{hyperref}

\theoremstyle{plain}
\newtheorem{theorem}{Theorem}[section]
\newtheorem{proposition}[theorem]{Proposition}

\newtheorem{corollary}[theorem]{Corollary}

\theoremstyle{definition}
\newtheorem{definition}[theorem]{Definition}

\newtheorem{remark}[theorem]{Remark}

\title{    A Route to Harris ergodicity for Non-Feller Markov Kernels}

\author{
Jean-Gabriel~Attali\thanks{Contact: jean-gabriel.attali@devinci.fr} \\
\small De Vinci Higher Education, De Vinci Research Center, Paris, France
}

\date{} 

\begin{document}
\maketitle

\begin{abstract}
For non-Feller Markov kernels satisfying a quasi-Feller factorization, the
compact-petite-set criterion gives petite compact sets under
\(\psi\)-irreducibility when the support of \(\psi\) has non-empty interior.
Thus, for coercive Lyapunov functions, the sublevel sets used in the
Harris--Lyapunov argument are petite. Under aperiodicity they are small. The
Hairer--Mattingly contraction theorem then yields geometric ergodicity in
weighted total variation under the usual geometric drift condition.
\end{abstract}

\section{Introduction}

Harris' theorem gives a standard route from irreducibility and Lyapunov
stability to convergence of Markov chains on general state spaces. In the
geometric case, one combines a Foster--Lyapunov drift condition with a
minorization condition on a suitable Lyapunov sublevel set; see Harris
\cite{Harris1956}, Nummelin \cite{Nummelin1984}, and Meyn--Tweedie
\cite{MeynTweedie2009}. Hairer and Mattingly
\cite{HairerMattingly2011} gave a direct contraction proof in weighted total
variation under these same two hypotheses.

For discontinuous stochastic recursions, the drift estimate is often easier to
check than the minorization on Lyapunov sublevel sets. This is the point
addressed here. The quasi-Feller condition introduced in \cite{Attali2004}
covers kernels which need not be Feller, but which factor through a Feller
kernel after a measurable map. Under \(\psi\)-irreducibility, provided that \(\operatorname{supp}\psi\) has
non-empty interior, the compact-petite-set criterion of \cite{Attali2004}
implies that compact sets are petite.

If the Lyapunov function is coercive, its sublevel sets are compact. Hence, for
quasi-Feller kernels satisfying the preceding irreducibility assumption, these
sublevel sets are petite. With aperiodicity, the usual petite-to-small
implication makes them small. The local hypothesis in the Hairer--Mattingly
theorem is then obtained from the quasi-Feller compact-petite criterion.

\section{Quasi-Feller kernels and compact petite sets}
\label{sec:qf}

Let \(S\) be a Polish space endowed with its Borel \(\sigma\)-field
\(\mathcal B(S)\). We denote by \(C_b(S)\) the space of bounded continuous
real-valued functions on \(S\), and by \(\mathcal P(S)\) the set of Borel
probability measures on \(S\).

Let \(P\) be a Markov transition kernel on \(S\). We write
\[
        Pf(x)=\int_S f(y)P(x,dy),
        \qquad f\in B_b(S),
\]
and similarly for the action of \(P\) on finite measures.

\subsection{Quasi-Feller kernels}

We recall the quasi-Feller condition introduced in \cite{Attali2004}. It is a
weak regularity assumption allowing \(P\) itself to be non-Feller, while
requiring that the discontinuity can be factorized through a measurable map.

\begin{definition}[Quasi-Feller kernel]
\label{def:qf}
A Markov kernel \(P\) on \(S\) is called quasi-Feller if there exist a Polish
space \(W\), a measurable map \(H:S\to W\), and a Markov kernel \(Q\) from
\(W\) to \(S\) such that:
\begin{enumerate}
    \item for every \(f\in C_b(S)\), the function
    \[
            Qf(w)=\int_S f(y)Q(w,dy)
    \]
    belongs to \(C_b(W)\);
    \item for every \(f\in C_b(S)\),
    \[
            Pf=Qf\circ H;
    \]
    \item if \(D_H\) denotes the set of discontinuity points of \(H\), then
\[
        Q(w,D_H)=0,
        \qquad
        w\in \bigcup_{K\subset S\ \mathrm{compact}} \overline{H(K)} .
\]
\end{enumerate}
\end{definition}

The Feller case corresponds to the special choice \(W=S\),
\(H=\operatorname{Id}\), and \(Q=P\).

Definition~\ref{def:qf}  allows discontinuous transition operators while preserving the
lower semicontinuity needed in compactness arguments. More precisely, if \(O\subset S\) is open, then
\(w\mapsto Q(w,O)\) is lower semicontinuous, and the condition
\(Q(w,D_H)=0\) ensures that the discontinuity set of \(H\) is not charged by
the one-step transition after factorization.

In what follows we use the compact-petite-set criterion of
\cite{Attali2004} in the form stated in
Theorem~\ref{thm:qf-compact-petite}. No additional stability assumption on
the iterates of the quasi-Feller factorization is needed beyond the hypotheses
of that theorem.

\subsection{Irreducibility and petite sets}

We use the standard notion of \(\psi\)-irreducibility.

\begin{definition}[\(\psi\)-irreducibility]
A Markov kernel \(P\) on \(S\) is \(\psi\)-irreducible if there exists a
\(\sigma\)-finite measure \(\psi\) on \((S,\mathcal B(S))\) such that, for
every \(A\in\mathcal B(S)\) with \(\psi(A)>0\) and every \(x\in S\),
\[
        \sum_{n\ge 0} P^n(x,A)>0.
\]
\end{definition}

The relevant local sets in the Harris theory are petite sets.

\begin{definition}[Petite set]
A measurable set \(C\subset S\) is petite if there exist a probability
distribution \(a=(a_n)_{n\ge1}\) on \(\mathbb N^*=\{1,2,\ldots\}\) and a
non-zero measure \(\nu\) on \(S\) such that
\[
        \sum_{n\ge1} a_n P^n(x,\cdot)\ge \nu(\cdot),
        \qquad x\in C .
\]
Equivalently, \(C\) is small for the averaged kernel
\[
        K_a=\sum_{n\ge1}a_nP^n .
\]
\end{definition}

Small sets correspond to the special case where the minorization holds at a
fixed time.

\begin{definition}[Small set]
A measurable set \(C\subset S\) is small if there exist \(m\ge1\) and a
non-zero measure \(\nu\) such that
\[
        P^m(x,\cdot)\geq \nu(\cdot),
        \qquad x\in C.
\]
\end{definition}

Every small set is petite. The converse is false in general; periodicity is
the basic obstruction. Under aperiodicity, the converse holds in the
\(\psi\)-irreducible setting, in the sense recalled in
Section~\ref{sec:harris}.

\subsection{Compact petite sets in the quasi-Feller setting}

The following result is the main input from \cite{Attali2004}.

\begin{theorem}[Compact petite sets for quasi-Feller chains]
\label{thm:qf-compact-petite}
Let \(P\) be a quasi-Feller Markov kernel on a Polish space \(S\). Assume that
\(P\) is \(\psi\)-irreducible and that
\[
        \operatorname{Int}(\operatorname{supp}\psi)\neq\varnothing .
\]

Then every compact set \(K\subset S\) is petite.
\end{theorem}

We recall the proof mechanism. The non-empty interior assumption allows one to
choose an open set \(O\subset \operatorname{supp}\psi\) with \(\psi(O)>0\).
Irreducibility gives accessibility of \(O\) from every point. The quasi-Feller
factorization then turns pointwise positivity of hitting probabilities of open
sets into local lower bounds, because \(w\mapsto Q(w,O)\) is lower
semicontinuous. Finally, compactness of \(H(K)\) permits a finite covering
argument, yielding a uniform minorization for an averaged kernel on \(K\).

Thus Theorem~\ref{thm:qf-compact-petite} gives compact petite sets without a
direct local minorization estimate. Without an open set of positive
\(\psi\)-mass, the above argument has no direct substitute.

\subsection{Lyapunov sublevel sets}

Let \(V:S\to[1,\infty)\) be measurable. We say that \(V\) is coercive if its
sublevel sets
\[
        C_R=\{x\in S: V(x)\le R\}
\]
are compact for all \(R<\infty\).

Assume that \(V\) is coercive and that \(P\) satisfies the geometric drift
condition
\[
        PV\le \gamma V+K,
        \qquad \gamma<1,\quad K<\infty .
\]
Then the sublevel sets required in the Hairer--Mattingly theorem,
\[
        C_R=\{V\le R\},
        \qquad R>\frac{2K}{1-\gamma},
\]
are compact. Hence, by Theorem~\ref{thm:qf-compact-petite}, they are petite
whenever \(P\) is quasi-Feller, \(\psi\)-irreducible, and
\(\operatorname{Int}(\operatorname{supp}\psi)\neq\varnothing\).

\section{Averaged kernels, aperiodicity and Harris contraction}
\label{sec:harris}

We recall the standard facts used below.

\subsection{Averaged kernels}

If \(C\) is petite, then for some
probability distribution \(a=(a_n)_{n\ge1}\) the averaged kernel
\[
        K_a=\sum_{n\ge1}a_nP^n
\]
has \(C\) as a small set.

The drift condition also passes to \(K_a\). If
\[
        PV\le \gamma V+K,\qquad \gamma<1,
\]
then, for \(n\ge1\),
\[
        P^nV\le \gamma^n V+K\frac{1-\gamma^n}{1-\gamma}.
\]
Consequently
\[
        K_aV
        \le
        \bar\gamma V+
        \frac{K(1-\bar\gamma)}{1-\gamma},
        \qquad
        \bar\gamma=\sum_{n\ge1}a_n\gamma^n<1.
\]
Thus \(K_a\) inherits a geometric Lyapunov drift from \(P\).

\subsection{Aperiodicity and the petite-to-small step}

Let \(P\) be a \(\psi\)-irreducible Markov kernel. Aperiodicity is the condition
which turns petite-set minorization into deterministic-time minorization. We use
the following standard consequence; see
\cite{Nummelin1984,MeynTweedie2009}.

\begin{proposition}[Petite sets are small under aperiodicity]
\label{prop:petite-small}
Let \(P\) be a \(\psi\)-irreducible and aperiodic Markov kernel. If \(C\) is
petite, then \(C\) is small. Equivalently, there exist an integer \(m\ge1\)
and a non-zero measure \(\nu\) such that
\[
        P^m(x,\cdot)\ge \nu(\cdot),
        \qquad x\in C .
\]
\end{proposition}

Combining Proposition~\ref{prop:petite-small} with the quasi-Feller
compact-petite-set criterion gives the following immediate consequence.

\begin{corollary}
\label{cor:qf-compact-small}
Let \(P\) be quasi-Feller, \(\psi\)-irreducible and aperiodic. Assume that
\[
        \operatorname{Int}(\operatorname{supp}\psi)\neq\varnothing .
\]
Then every compact set \(K\subset S\) is small.
\end{corollary}

Indeed, every compact set is petite by
Theorem~\ref{thm:qf-compact-petite}, and petite sets are small under
aperiodicity.

\subsection{The Hairer--Mattingly contraction theorem}

We recall the version of the Hairer--Mattingly theorem used below. Let \(P\)
be a Markov kernel on \(S\), and let \(V:S\to[0,\infty)\) be a measurable
Lyapunov function. Assume that there exist constants \(\gamma\in(0,1)\) and
\(K<\infty\) such that
\[
        PV\le \gamma V+K .
\]
Let
\[
        C_R=\{x\in S:V(x)\le R\},
        \qquad
        R>\frac{2K}{1-\gamma}.
\]
Assume that \(C_R\) is small, i.e. that there exist \(\alpha\in(0,1)\) and a
probability measure \(\nu\) such that
\[
        P(x,\cdot)\ge \alpha\nu(\cdot),
        \qquad x\in C_R .
\]

For \(\beta>0\), define
\[
        \|\varphi\|_\beta
        =
        \sup_{x\in S}
        \frac{|\varphi(x)|}{1+\beta V(x)}
\]
and the associated dual distance on probability measures
\[
        \rho_\beta(\mu_1,\mu_2)
        =
        \sup_{\|\varphi\|_\beta\le1}
        \int_S \varphi\,d(\mu_1-\mu_2).
\]
Equivalently,
\[
        \rho_\beta(\mu_1,\mu_2)
        =
        \int_S (1+\beta V(x))\,|\mu_1-\mu_2|(dx),
\]
so \(\rho_\beta\) is a weighted total variation distance.

\begin{theorem}[Hairer--Mattingly]
\label{thm:HM}
Under the preceding drift and small-set assumptions, there exist
\(\beta>0\) and \(\bar\gamma\in(0,1)\) such that
\[
        \rho_\beta(\mu_1P,\mu_2P)
        \le
        \bar\gamma\,\rho_\beta(\mu_1,\mu_2)
\]
for all probability measures \(\mu_1,\mu_2\) with finite \(V\)-moment.
Consequently, \(P\) admits a unique invariant probability measure \(\pi\), and
there exist constants \(C<\infty\) and \(\rho\in(0,1)\) such that
\[
        \|P^n(x,\cdot)-\pi\|_V
        \le
        C(1+V(x))\rho^n,
        \qquad n\ge0 .
\]
\end{theorem}

The proof in \cite{HairerMattingly2011} is direct. The parameter \(\beta\) is
chosen so that the Lyapunov drift gives contraction away from the sublevel
set, while the minorization gives contraction inside the sublevel set. No
excursion decomposition or return-time estimates are used.

\section{From quasi-Feller petite sets to Harris contraction}
\label{sec:main}

Let \(S\) be a Polish space and let \(P\) be a Markov kernel on \(S\). Let
\(V:S\to[1,\infty)\) be a measurable coercive function, so that
\[
        C_R=\{x\in S:V(x)\le R\}
\]
is compact for every \(R<\infty\). Assume that \(P\) satisfies the geometric
drift condition
\begin{equation}\label{eq:drift}
        PV\le \gamma V+K,
        \qquad \gamma\in(0,1),\quad K<\infty .
\end{equation}

\subsection{The averaged kernel}

We first record the conclusion available without aperiodicity.

\begin{proposition}[Averaged kernel]
\label{prop:averaged}
Assume that \(P\) is quasi-Feller and \(\psi\)-irreducible, with
\[
        \operatorname{Int}(\operatorname{supp}\psi)\neq\varnothing .
\]
Assume also that \(P\) satisfies the coercive geometric drift
\eqref{eq:drift}. Then every Lyapunov sublevel set
\(C_R=\{V\le R\}\) is petite.

Consequently, if \(R>2K/(1-\gamma)\), there exists a probability distribution
\(a=(a_n)_{n\ge1}\) on \(\mathbb N^*\) such that the averaged kernel
\[
        K_a=\sum_{n\ge1}a_nP^n
\]
has \(C_R\) as a small set. Moreover \(K_a\) satisfies a geometric Lyapunov
drift for which the same threshold \(R>2K/(1-\gamma)\) is sufficient in the
Hairer--Mattingly theorem. Hence \(K_a\) is geometrically ergodic in weighted
total variation.
\end{proposition}

\begin{proof}
Since \(V\) is coercive, \(C_R\) is compact; hence it is petite by
Theorem~\ref{thm:qf-compact-petite}. Thus, for some probability distribution
\(a=(a_n)_{n\ge1}\) and some non-zero measure \(\nu\),
\[
        K_a(x,\cdot)
        =
        \sum_{n\ge1}a_nP^n(x,\cdot)
        \ge \nu(\cdot),
        \qquad x\in C_R .
\]
Therefore \(C_R\) is small for \(K_a\).

It remains only to check the drift. From \eqref{eq:drift},
\[
        P^nV
        \le
        \gamma^nV
        +
        K\frac{1-\gamma^n}{1-\gamma},
        \qquad n\ge1.
\]
Thus
\[
        K_aV
        \le
        \bar\gamma V
        +
        \frac{K(1-\bar\gamma)}{1-\gamma},
        \qquad
        \bar\gamma=\sum_{n\ge1}a_n\gamma^n<1 .
\]
The corresponding Hairer--Mattingly threshold is
\[
        \frac{2}{1-\bar\gamma}
        \frac{K(1-\bar\gamma)}{1-\gamma}
        =
        \frac{2K}{1-\gamma}.
\]
Hence the assumption \(R>2K/(1-\gamma)\) is sufficient for \(K_a\), and the
Hairer--Mattingly theorem gives the claimed geometric ergodicity.
\end{proof}

\subsection{The original kernel}

Aperiodicity converts the averaged minorization into a deterministic-time
minorization.

\begin{theorem}[Geometric ergodicity of the original chain]
\label{thm:main}
Assume that \(P\) is quasi-Feller, \(\psi\)-irreducible and aperiodic, with
\[
        \operatorname{Int}(\operatorname{supp}\psi)\neq\varnothing .
\]
Assume also that \(P\) satisfies the coercive geometric drift
\eqref{eq:drift}. Then \(P\) admits a unique invariant probability measure
\(\pi\). Moreover, there exist constants \(C<\infty\) and \(\rho\in(0,1)\)
such that
\begin{equation}\label{eq:geometric-convergence}
        \|P^n(x,\cdot)-\pi\|_V
        \le
        C(1+V(x))\rho^n,
        \qquad x\in S,\quad n\ge0 .
\end{equation}
Thus \(P\) is geometrically ergodic in weighted total variation.
\end{theorem}

\begin{proof}
Choose \(R>2K/(1-\gamma)\). Since \(V\) is coercive,
\(C_R=\{V\le R\}\) is compact. By
Theorem~\ref{thm:qf-compact-petite}, \(C_R\) is petite. Since \(P\) is
aperiodic, Proposition~\ref{prop:petite-small} implies that \(C_R\) is small.
Thus there exist an integer \(m\ge1\) and a non-zero measure \(\nu\) such that
\[
        P^m(x,\cdot)\ge \nu(\cdot),
        \qquad x\in C_R .
\]

If \(m=1\), the Hairer--Mattingly theorem applies directly to \(P\), using the
drift condition and the smallness of \(C_R\). This gives
\eqref{eq:geometric-convergence}.

Assume now that \(m>1\). We apply the same argument to the skeleton \(P^m\).
Iterating the drift condition gives
\[
        P^mV
        \le
        \gamma^m V
        +
        K\frac{1-\gamma^m}{1-\gamma}.
\]
Thus \(P^m\) satisfies a geometric drift, with drift coefficient \(\gamma^m\)
and constant \(K(1-\gamma^m)/(1-\gamma)\), and \(C_R\) is small for \(P^m\).

Moreover, the same Lyapunov level is sufficiently large for the skeleton.
Indeed, the Hairer--Mattingly threshold associated with this drift is
\[
        \frac{
        2K(1-\gamma^m)/(1-\gamma)
        }{
        1-\gamma^m
        }
        =
        \frac{2K}{1-\gamma}.
\]
Hence the original choice \(R>2K/(1-\gamma)\) already ensures that \(C_R\) is
a sufficiently large sublevel set for \(P^m\). Theorem~\ref{thm:HM} therefore
applies to \(P^m\). Hence \(P^m\) has a unique invariant probability measure
\(\pi\) with finite \(V\)-moment, and
\[
        \|(P^m)^q(x,\cdot)-\pi\|_V
        \le
        C_0(1+V(x))\rho_0^q,
        \qquad q\ge0 ,
\]
for some \(C_0<\infty\) and \(\rho_0\in(0,1)\).
The measure \(\pi\) is invariant for \(P\). Indeed, if \(\pi P^m=\pi\), then
\[
        (\pi P)P^m = \pi P^{m+1} = (\pi P^m)P = \pi P .
\]
Thus \(\pi P\) is also invariant for \(P^m\). By uniqueness of the invariant
probability measure of \(P^m\), one has \(\pi P=\pi\).

It remains to pass from the skeleton to all iterates. Write
\[
        n=qm+r,
        \qquad 0\le r<m.
\]
For each \(r\in\{0,\ldots,m-1\}\), the drift condition gives
\[
        P^rV\le A_r(1+V)
\]
for some finite constant \(A_r\). Hence the finitely many operators \(P^r\)
are bounded in the weighted total variation norm. Applying \(P^r\) to the
skeleton estimate and absorbing the finite constants into \(C\) yields
\eqref{eq:geometric-convergence}.
\end{proof}

\begin{remark}
Aperiodicity is used only to pass from petite sets to small sets. Without it,
the conclusion remains available for a suitable averaged kernel \(K_a\), but
not necessarily for the deterministic iterates of \(P\).
\end{remark}

\section{Example}
\label{sec:examples}

We record a simple autoregressive example. Examples of ARCH type and
Markov-switching type are discussed in \cite{Attali2004}.

Let \(f:\mathbb R\to\mathbb R\) be Riemann-integrable on compact intervals and
assume that, for some \(a\in(0,1)\) and \(B<\infty\),
\[
        |f(x)|\le a|x|+B .
\]
Let \((\xi_n)_{n\ge1}\) be i.i.d. with density \(p\), where
\[
        p>0\quad \lambda\text{-a.e.},
        \qquad
        \mathbb E|\xi_1|<\infty .
\]
Consider
\[
        X_{n+1}=f(X_n)+\xi_{n+1}.
\]
If \(f\) is discontinuous, then the transition kernel \(P\) need not be Feller.

The quasi-Feller factorization is
\[
        P(x,\cdot)=Q(f(x),\cdot),
        \qquad
        Q(z,A)=\int_{\mathbb R}\mathbf 1_A(z+u)p(u)\,du .
\]
The kernel \(Q\) is Feller and  the
discontinuity set \(D_f\) has Lebesgue measure zero. Since
\(Q(z,\cdot)\ll\lambda\) for every \(z\), \(Q(z,D_f)=0\). Hence \(P\) is
quasi-Feller.

The chain is Lebesgue-irreducible and aperiodic, since \(p>0\)
\(\lambda\)-a.e. Finally, with \(V(x)=1+|x|\),
\[
        PV(x)
        =
        1+\mathbb E|f(x)+\xi_1|
        \le
        1+a|x|+B+\mathbb E|\xi_1|
        \le
        aV(x)+K
\]
for some finite \(K\). Since \(V\) is coercive, all assumptions of
Theorem~\ref{thm:main} are satisfied.

\bibliography{yet}

\end{document}